\documentclass[10pt,a4paper,reqno]{amsart}

\usepackage{amsmath,amssymb,amsfonts,amsthm,enumerate,array,color,mathtools}
\usepackage{bbm} 

\usepackage{setspace}
\theoremstyle{plain}\newtheorem{theorem}{Theorem}
\newtheorem{corollary}[theorem]{Corollary}
\newtheorem{prop}[theorem]{Proposition}
\newcommand{\veps}{\varepsilon}
\newcommand{\Var}{\mathop{\mathrm{Var}}}
\renewcommand\d{\mathrm{d}}
\newcommand{\ind}{\mathbbm{1}} 

\title[Moments of general time dependent branching processes]
{Moments of general time dependent branching processes with applications} 
\author{Tam\'as F.~M\'ori}
 \address{Department of Probability Theory and Statistics\\
 E\"otv\"os Lor\'and University\\
P\'azm\'any P.~s. 1/C, 
 H-1117 Budapest, Hungary}
 \email{mori@math.elte.hu}
\author{S\'andor Rokob}
\address{Department of Stochastics\\
 Budapest University of Technology and Economics\\
 Egry J.~u. 1, 
 H-1111 Budapest, Hungary}
 \email{roksan@math.bme.com}

\thanks{This work was supported by the Hungarian
  National Research, Development and Innovation Office NKFIH, grant
  number K 125569 (T.F.~M\'ori), and the \'UNKP-17-2 New National Excellence
  Program of the Ministry of Human Capacities (S.~Rokob).}
  \thanks{ Corresponding author: T.~F.~M\'ori.}


\date{07 May 2019}

\begin{document}

\begin{abstract}
In this paper, we give sufficient conditions for a Crump--Mode--Jagers process to be 
bounded in $L_k$ for a given $k>1$. This result is then applied to a recent random 
graph process motivated by pairwise collaborations and driven by time-dependent 
branching dynamics. We show that the maximal degree has the same rate of increase
as the degree process of a fixed vertex.
\end{abstract}

\maketitle
\thispagestyle{empty}


\section{Introduction}

The present paper is devoted to the study of some properties of Crump--Mode--Jagers 
(CMJ) processes. More precisely, we are interested in conditions sufficient for a suitably 
scaled CMJ process (counted by appropriate random characteristics) to have finite 
$k{\text{th}}$ moments, uniformly in $t \geq 0$. Many properties of CMJ 
processes have already been studied in full details, but little is known about the asymptotics 
of higher moments, apart from the variance, which was used to prove convergence in 
probability before new techniques made it possible to prove almost sure results. 

Our research is motivated by the problem of 
describing the asymptotic behaviour of the maximal degree in a recently introduced 
random graph process evolving in continuous time \cite{MR17}. That graph process, called 
the collaboration model, is driven by time-dependent branching dynamics. Our results are 
strong enough to attack the order of magnitude of the maximal degree in the motivating 
random graph process; however, the present paper can only be considered the first 
step towards a general theory, as these results are surely open to improvement.

CMJ processes or, in other words, general time-dependent branching processes, have been present in
the theory of stochastic processes as well as in modelling and various applications for half a century.
The classical theory of a.s. convergence of CMJ processes counted with random characteristics was 
completed by the middle of the eighties \cite{JN}, but several extensions and generalizations are still 
at the forefront of research. For instance, in biological applications it is natural to suppose that a process 
describing the increase of a certain population cannot grow indefinitely. This led to the introduction of
population-size-dependent branching processes \cite{Jagers96,JK}. In the classical theory the
existence of the so called Malthusian parameter is assumed, which implies exponential growth in the
supercritical case, but recently efforts are made to develop a theory of CMJ processes in the setting where 
no Malthusian parameter exist, hence the process gows faster than exponential \cite{KOM}, etc.

Among many different fields of applications, the CMJ processes can really be helpful in the analysis of 
random graph models motivated by real life networks. For example, one can consider asymptotic 
properties of random tree growth models driven by preferential attachment dynamics \cite{BHAMIDI, RT, RTV}, 
its generalizations \cite{DEREICH} or the behaviour of spreading processes on different random graph models 
\cite{KOM}. Without attempting to be comprehensive we also mention \cite{AGS,Devroye,HJ,JL}.
 In all of these cases this general branching process enters the picture in a very natural way.

In the first part of this article, we prove that the proper $L_k$ boundedness of the 
random characteristic stochastic processes associated with the CMJ process guarantees the 
required $L_k$ boundedness of the CMJ process itself. Then we turn to the application of 
our results to the collaboration model, and we conclude that in our random graph process 
the maximal degree has the same asymptotic order of magnitude as the degree of an 
arbitrary vertex.

Although the general theory of CMJ processes is widely known (a comprehensive 
summary can be found e.g. in \cite[Ch.6]{Jagers} or \cite{HJV}), for the sake of 
convenience in the sequel we briefly and informally recollect some basic notions and properties.

We consider a supercritical CMJ process given by the triple 
$(\lambda,\,\xi,\,\phi)$, where $\lambda$ and $\xi(t)$, $t\ge 0$, are 
the random life span, and the reproduction process, resp., of a generic individual, 
and $\phi(t)$,  $t\ge 0$, is a nonnegative random characteristic. Here $\xi(t)$ 
is the number of offspring up to and including time $t$. By appointment, 
$\xi(t)=\xi(t\wedge\lambda)$, that is, no reproduction takes place after 
death.  In addition, $\xi(t)=\phi(t)=0$ for $t<0$. 
We also suppose that the point process $\xi$ is not a lattice. 

The process of interest is the following.
At time $0$, a single individual, called the ancestor, starts its life. It  produces offspring according
to the random point process $\xi(t)$, up to its death time $\lambda$. All its descendants, 
independently of all other individuals, produce their own offspring, and so on. At every instant we
evaluate each individual by a random characteristic and consider the sum of individual values.
We remark that the condition of starting with a single ancestor can obviously be relaxed. In case of $n$ ancestors
we are dealing with $n$ i.i.d. copies of the single ancestor CMJ process.

Many of the well-known properties of discrete time Galton--Watson processes are inherited 
by the far more general CMJ processes. One of these attributes is self-similarity, which is nothing 
else than the possibility of representing the process as the superposition of a random number 
of retarded independent copies of the original process. Since this feature will be used in the sequel, 
let us give it in more details, introducing some more notations.

The process we are studying is
\[
Z^{\phi}(t)=\sum_\ell \phi_\ell(t-\tau_\ell),
\]
where the summation runs over all individuals $\ell$ ever born, including 
the ancestor, $\tau_\ell$ is the birth time of individual $\ell$ (zero for the 
ancestor). Particularly, for $\phi(t)=\ind(t\ge 0)$ we obtain $Z^\phi(t)=T(t)$, 
the number of individuals, dead or alive, born up to and including time $t$.
(Here and in the sequel $\ind(\,.\,)$ stands for the indicator of the event in brackets.)
Thus we have
\begin{equation*}
T(t)=\sum_{\ell} \ind(\tau_{\ell}\le t).
\end{equation*}
Labelling individuals in chronological order we can write
\begin{equation}\label{zphi}
Z^{\phi}(t)=\sum_{i=0}^{T(t)-1}\phi_i(t-\tau_i).
\end{equation}
Clearly,
\begin{equation}\label{alap}
Z^\phi(t)=\phi_0(t)+\sum_{i=1}^{\xi_0(t)}Z_i^{\phi}(t-\sigma_i),
\end{equation}
where $\sigma_1,\,\sigma_2,\,\dots$ are the birth times (in increasing order) 
of the children of the ancestor, the corresponding reproduction processes 
and random characteristics are denoted by $\xi_i$ and $\phi_i$, 
$i=1,\,2,\,\dots$, while $\xi_0$ and $\phi_0$ belong 
to the ancestor. Each child starts its own CMJ process counted by the random 
characteristic; they are denoted by $Z_i^\phi$, $i=1,\,2,\,\dots$.

Finally, let $\mu(t)=E\xi(t)$; the Lebesgue--Stiletjes measure it generates is called the 
reproduction measure. We assume the existence of 
a positive $\alpha$, called \emph{the Malthusian parameter} for which 
\begin{equation}\label{Malthus}
\int_0^\infty e^{-\alpha t}\mu(\d t)=1.
\end{equation}
In this case the CMJ process is said to be in the so-called supercritical regime, and on the event of 
non-extinction, the probability of which is strictly less than $1$, its growth rate is exponential with 
exponent $\alpha$.

\section{$L_k$ bounds for CMJ processes}

In this section, we will first present sufficient conditions for a CMJ process $Z^{\phi}$ 
to fall in $L_k$ at a fixed time $t$, then we turn to the whole process, and try
to find conditions under which the suitably normed CMJ process is bounded in $L_k$.

Following Nerman \cite{Nerman}, let us introduce
\[
{}_\alpha\xi(t)=\int_0^t e^{-\alpha s}\xi(\d s), \quad
_{\alpha}\mu(t)=E\big({}_\alpha\xi(t)\big)=\int_0^t e^{-\alpha s}\mu(\d s),
\]
then $_\alpha\xi(\infty)=\int_0^\infty e^{-\alpha t}\xi(\d t)$, thus 
${}_{\alpha}\mu(\infty)=E(_\alpha\xi(\infty))=1$ by \eqref{Malthus}. Here 
and in the sequel the domain of integration is always closed.

Instead of dealing with $k$th moments, we rather focus on $L_k$ norms. 
For a random variable $U$ let
\[
{\|U\|}_k=\left[E(|U|^k)\right]^{1/k}.
\]
By using these, we can state and prove our first proposition, which deals with the $L_k$ 
boundedness of the process $T(t)$.

\begin{prop}\label{T(t)}
Let $t>0$ and $k>1$ be fixed. Suppose ${\|\xi(0)\|}_k<1$ and ${\|\xi(t)\|}_k<\infty$. 
Then ${\|T(t)\|}_k<\infty$.
\end{prop}

\begin{proof}
The proof will be performed by coupling the original CMJ process with another, 
more tractable one where birth events are more frequent. This will be done in 
two steps. 

Clearly, if the reproduction process $\xi(s)$ is replaced with another one $\xi'(s)$, 
such that $\xi'(s)\ge\xi(s)$ for every $s\ge 0$, then the two CMJ processes can 
be coupled in such a way that the total number of individuals 
born up to and including $s$ cannot decrease: $T'(s)\ge T(s)$ for every 
$s\ge 0$. 

Since ${\|\xi(s)\|}_k$ is right continuous by the monotone convergence theorem,
one can find a sufficiently large integer $N$ such that ${\|\xi(\veps)\|}_k<1$ 
for $\veps=t/N$.
Now, let
\[
\xi'(s)=\begin{cases*}
\xi(\veps) & if $0\le s<\veps$, \\
\xi(t) & if $\veps\le s <t$, \\
\xi(s) & if $t\le s$;
\end{cases*}
\]
that is, all birth events taking place in the interval $[0,\veps]$ are advanced to
time zero, and birth events between $\veps$ and $t$ are advanced to time $\veps$.
In this way $\xi'(s)\ge \xi(s)$ for every nonnegative $s$. Note that 
$E\xi'(0)\le {\|\xi'(0)\|}_k={\|\xi(\veps)\|}_k<1$.

Note that $\xi'(0)$ is the (random) number of children born at the 
very moment of the birth of their mother, the ancestor. Each of them may 
immediately give birth to new individuals, which are the grandchildren of the ancestor, 
and so on. The successive generations, all starting their lifes at the same moment, 
form a Galton--Watson process with offspring distribution $\xi'(0)$.
Since $\mu'(0)= E\xi'(0)<1$, this process is subcritical, and the number $\zeta$ of 
all individuals ever born is almost surely finite, as it has a finite mean
$[1-\mu'(0)]^{-1}$. All these descendants can be considered as siblings of the 
ancestor; in other words, every individual who was born simultaneously with its parent 
will be considered the sibling of its farthest coeval ancestor. In this way we obtain a new
CMJ process with the same life length $\lambda$ but with multiple ancestors and a different reproduction 
process: $\xi'$ is replaced by $\xi''$, which can be obtained from $\xi'$ as follows.
$\xi''(0)=0$  (that is, no offspring at time zero), and whenever an individual is born 
according to the original process $\xi'$, the new process $\xi''$ gives birth to a random 
number of individuals, equidistributed with $\zeta$. Thus, 
\[
\xi''(s)=\sum_{i=1}^{\xi'(s)-\xi'(0)}\zeta_i,
\]
where $\zeta_1,\,\zeta_2,\,\dots$ are i.i.d. copies of $\zeta$. The CMJ process defined by
$\xi'$ can be equivalently replaced with a CMJ process that starts with $\zeta$ ancestors and
is governed by the reproduction process $\xi''$, in the sense that the total number of individuals 
born up to time $t$ is the same for the two processes: $T'(t)=T''(t)$.

We can formalize the construction of $\xi''$ by using family trees (phylogenetic trees).
In a CMJ process every individual can be labelled with a finite string of positive integers
based on the parent-child relations. Usually the ancestor is associated with the empty 
string, but if we allow for the possibility of several ancestors, we had better to label 
ancestors with consecutive positive integers (length $1$ strings), starting with $1$. 
Then the $j$th child (in the order of birth) of the individual labelled with $A=(a_1,a_2,\dots,a_k)$ 
is labelled with $(A, j)$. That is, individual $(a_1,\dots,a_k)$ is the $a_k$th child of 
individual $(a_1,\dots,a_{k-1})$, who is the $a_{k-1}$st child of its parent, and so on.
Consider the family tree of the CMJ process generated by the reproduction process $\xi'$,
and in the labels color every number red that corresponds to an individual born 
at the very same moment with its parent (all the other numbers are left black). 
Now, let us eliminate red numbers successively.
Order the labels with shortlex ordering, i.e., firstly, according to the lengths of the strings,
the shorter always comes before the longer, then strings of the same length are ordered 
lexicographically. The algorithm consists of the following steps repeated iteratively.
\begin{enumerate}[(i)]
\item \textit{Recoloring by relabelling}. Take the first label that is not entirely black. Suppose its length is $k$.
It contains exactly one red number and it is at the last position: $A=(B,a,1)$, where $B$ can 
be empty. Let 
\[
b=1+\max\{x: (B,x) \text{ was born at the same time as individual }A\},
\]
then $b>a$. In all strings beginning with $(B,x)$ with $x\ge b$ change 
$x$ to $x+1$ in the $(k-1)$st position with keeping the colors. Then replace $A$ with the purely black 
$A'=(B,b)$ in all strings containing $A$ as a prefix.
\item \textit{Compression}.
Change all remaining strings of the form $(B,a,x)$ to $(B,a,x-1)$ in all their occurrences as prefixes.
\end{enumerate}

We emphasize that this algorithm does not change the birth times, only the family relations, thus
$T'(t)=T''(t)$. As there can be infinitely many individuals same age as their parents, the algorithm 
does not necessarily stop in finitely many steps, but fixing an arbitrarily large bound $N$, eventually 
no labels of length less than $N$ would change.

Next, we will show that ${\|\zeta\|}_k<\infty$. 

To this end, first note that a random sum $\sum_{i=1}^N Y_i$, where the summands $Y_i$ 
are i.i.d. and $N$ is independent of them, has finite $k$th moment provided 
so do $Y_1$ and $N$. Indeed, by the power mean inequality we have
\begin{multline*}
E\bigg(\Big|\sum_{i=1}^N Y_i\Big|^k\bigg)\le E\bigg(N^{k-1}\sum_{i=1}^N
|Y_i|^k\bigg)=E\Bigg[ E\bigg(N^{k-1}\sum_{i=1}^N |Y_i|^k\biggm| N \bigg)
\Bigg]\\
= E\big(N^k E|Y_1|^k\big)=E(N^k) E|Y_1|^k <\infty.
\end{multline*}
Let $G_n$ denote the size of generation $n$ in the Galton--Watson process above,
$n=0,\,1,\,\dots$, $G_0=1$. Then $G_n$ is a random sum with $N=G_{n-1}$, 
and we get by induction that ${\|G_n\|}_k\le \big({\|\xi(0)\|}_k\big)^n$. Hence
\[
{\|\zeta\|}_k=\bigg\|\sum_{j=0}^\infty G_j\bigg\|_k\le \sum_{j=1}^\infty {\|G_j\|}_k
\le \frac{1}{1- {\|\xi(0)\|}_k}<\infty.
\]

For $s\le t$, the process $T''(s)$ behaves like a Galton--Watson process with
$0$-th generation $\zeta$, and offspring distribution 
\[
\xi''(\veps)=\sum_{i=1}^{\xi(t)-\xi(\veps)}\zeta_i,
\]
where the discrete process is embedded in continuous time with intergeneration 
time $\veps$. Let us redefine $G_i$, $i=0,\,1,\,\dots$ as the generation sizes
of this latter Galton--Watson process, then $T''(t)=G_0+G_1+\dots+G_N$.
This time we have
\[
{\|G_i\|}_k=\big({\|\xi''(\veps)\|}_k\big)^{i+1}\le
\big({\|\zeta\|}_k {\|\xi(t)-\xi(\veps)\|}_k\big)^{i+1}\le
\big({\|\zeta\|}_k {\|\xi(t)\|}_k\big)^{i+1}<\infty,
\]
therefore ${\|T''(t)\|}_k<\infty$, thus the proof is completed.
\end{proof}

Next, let us extend Proposition \ref{T(t)} to a more general class of random characteristics.

\begin{prop}\label{Zphi}
Let $t>0$ be fixed, and $Y=\sup_{s\le t}\phi(s)$. If ${\|Y\|}_k<\infty$ 
and ${\|T(t)\|}_k<\infty$, then ${\|Z^{\phi}(t)\|}_k<\infty$.
\end{prop}

\begin{proof}
Let $\mathcal F_n$ be the sigma-field generated by the processes 
$(\xi_i,\phi_i)$, $0\le i\le n$, and $Y_i=\sup_{s\le t}\phi_i(s)$, $i=0,\,1,\,\dots$. 
Clearly, they are i.i.d. random variables, and adapted to the filtration 
$(\mathcal F_n,\,n\ge 0)$. Moreover, $\nu=T(t)-1$ is a stopping time, for 
the event $\{\nu\le n\}=\{\tau_{n+1}>t\}$ is determined by 
$\xi_0,\,\xi_1,\,\dots,\xi_n$.

Obviously,
\[
Z^{\phi}(t)=\sum_{\tau_i\le t}\phi_i(t-\tau_i)\le \sum_{\tau_i\le t}Y_i=
\sum_{i=0}^{T(t)-1}Y_i.
\]
Let $\widetilde Y_i=Y_i-EY$, and $M_n=\sum_{i=0}^{\nu\wedge n}
\widetilde Y_i$, where $\wedge$ stands for minimum. Then $(M_n,\,
\mathcal F_n,\, n\ge 0)$ is a convergent martingale with differences 
$d_n=\widetilde Y_n \ind(n\le\nu)$. Clearly,
\[
M_\infty=\lim_{n\to\infty}M_n=\sum_{i=0}^{\nu}\widetilde Y_i=
\sum_{i=0}^{T(t)-1}Y_i - T(t)EY.
\]
Introduce the conditional variance process 
\begin{multline*}
A_n=E\big(d_0^2\big)+\sum_{i=1}^n E\big(d_i^2\bigm|\mathcal F_{i-1}\big)
=E\big(\widetilde Y_0^2\big)+\sum_{i=1}^n E\big(\widetilde Y_i^2\bigm|
\mathcal F_{i-1}\big) \ind(i\le \nu)\\
= (1+n\wedge \nu)E\big(\widetilde Y_0^2\big)=
\big[T(t)\wedge(n+1)\big]\Var(Y).
\end{multline*}
It converges to $A_{\infty}=T(t)\Var(Y)$.

By the Burkholder--Rosenthal inequality \cite{B} we have
\[
{\|M_\infty\|}_k\le C_k\left( \big\|A_\infty^{1/2}\big\|_k+\Bigg\|\bigg(
\sum_{i=0}^\infty |d_i|^k\bigg)^{\!1/k}\Bigg\|_k\right).
\]

Here $ \big\|A_\infty^{1/2}\big\|_k=\Var(Y)^{1/2}{\|T(t)\|}_{k/2}$, and
\[
\Bigg\|\bigg(\sum_{i=0}^\infty |d_i|^k\bigg)^{\!1/k}\Bigg\|_k =
\Bigg[E\bigg(\sum_{i=0}^\infty |d_i|^k\bigg)\Bigg]^{\! 1/k}
=\Bigg[E\bigg(\sum_{i=0}^\infty |\widetilde Y_i|^k \ind(i\le \nu)\bigg)\Bigg]^{\! 1/k}.
\]
Since $\widetilde Y_i$ and $ \ind(i\le \nu)$ are independent, we get
\[
\Bigg\|\bigg(\sum_{i=0}^\infty |d_i|^k\bigg)^{\!1/k}\Bigg\|_k
=\Bigg[E\big( |\widetilde Y_0|^k\big)\sum_{i=0}^\infty P(i\le \nu)\Bigg]^{\! 1/k}
=\big\|\widetilde Y_0\big\|_k \big(ET(t)\big)^{1/k}.
\]
The conditions of Proposition \ref{Zphi} imply that ${\|M_\infty\|}_k<\infty$.
Since $Z^{\phi}(t)\le M_{\infty}+T(t)EY$, we finally conclude that
${\|Z^{\phi}(t)\|}_k<\infty$, as needed.
\end{proof} 

It would be desirable to take the supremum out of the norm. At the moment we 
can only do it by requiring more of  $T(t)$.

\begin{prop}\label{Zphi2}
 Suppose $\sup_{s\le t}{\|\phi(s)\|}_k <\infty$ and 
${\|T(t)\|}_p<\infty$ for some $p>k$. Then ${\|Z^{\phi}(t)\|}_k<\infty$.
\end{prop}

\begin{proof}
$\displaystyle 
Z^{\phi}(t)=\sum_{i=0}^{\infty}\phi_i(t-\tau_i)\ind(\tau_i\le t)$,
therefore we obtain
\[
{\|Z^{\phi}(t)\|}_k \le\sum_{i=0}^{\infty}{\|\phi_i(t-\tau_i)\ind(\tau_i\le t)\|}_k\\
=\sum_{i=0}^{\infty}\Big(E\big[\phi_i(t-\tau_i)^k \ind(\tau_i\le t)\big]\Big)^{1/k}.
\]
Take conditional expectation given $\tau_i$ in the terms of the sum on the right-hand side.
By the independence of $\phi_i$ and $\tau_i$ we get
\begin{multline*}
E\big[\phi_i(t-\tau_i)^k \ind(\tau_i\le t)\big]=
E\big[E\big(\phi_i(t-\tau_i)^k\bigm|\tau_i\big) \ind(\tau_i\le t)\big]\\
\le E\Big[\sup_{s\le t}E\big(\phi_i(s)^k\big) \ind(\tau_i\le t)\Big]
=\sup_{s\le t}E\big(\phi_i(s)^k\,P(\tau_i\le t).
\end{multline*}
Hence, 
\[
Z^{\phi}(t)\le \sup_{s\le t}{\|\phi_i(s)\|}_k \sum_{i=0}^\infty 
\big[P(\tau_i\le t)\big]^{1/k}=\sup_{s\le t}{\|\phi_i(s)\|}_k 
\sum_{i=0}^\infty \big[P(T(t)> i)\big]^{1/k}.
\]
Let $p>k$. By the H\"older inequality we have
\begin{align*}
\sum_{i=1}^{\infty}\big[P(T(t)\ge i)\big]^{1/k}&=\sum_{i=1}^{\infty}
\big[ i^{\,p-1} P(T(t)\ge i)\big]^{1/k}{\strut i}^{\,-(p-1)/k}\\
&\le \bigg[ \sum_{i=1}^{\infty} i^{\,p-1} P(T(t)\ge i)\bigg]^{\tfrac{1}{k}}
\  \bigg[ \sum_{i=1}^{\infty}{\strut }i^{\,-\tfrac{p-1}{k-1}} \bigg]^{\tfrac{k-1}{k}}\\
&\le \text{Const.}\big({\|T(t)\|}_p\big)^{p/k}.
\end{align*}
This completes the proof.
\end{proof} 

\begin{theorem}\label{thm1}
Let $k$ be a positive integer. Suppose ${\|Z^\phi(t)\|}_k<\infty$ for every 
$t\ge 0$; furthermore, $A=:{\|{}_\alpha\xi(\infty)\|}_k<\infty$, and
$B=:\sup_{t\ge 0}e^{-\alpha t}{\|\phi(t)\|}_k<\infty$. Then
\begin{equation*}
\sup_{t\ge 0} e^{-\alpha t}{\|Z^\phi(t)\|}_k<\infty.
\end{equation*}
\end{theorem}

\begin{proof}
Let $t_1\le t_2\le\dots\le t_k$, and $t=t_1$, $h_i=t_{i+1}-t_i$, 
$i=1,\,\dots, k-1$. Introduce
\[
M_k(t,\mathbf{h})=E\left[\prod_{i=1}^k e^{-\alpha t_i} 
Z^{\phi}(t_i)\right].
\]
Particularly, $M_k(t,\mathbf{0})=\Big(e^{-\alpha t}{\|Z^{\phi}(t)\|}_k\Big)^k<\infty$.
By H\"older's inequality we have 
\[
M_k(t,\mathbf{h})\le\Big(M_k(t_1,\mathbf{0})\dots M_k(t_k,\mathbf{0})\Big)^{\!1/k}
<\infty.
\]

We will prove that $M_k(t,\mathbf{h})\le C_k$ for every $t$ 
and $\mathbf{h}$, where the upper bound $C_k$ may depend on 
$\phi$. We do it by induction over $k$.

For $k=1$ it follows easily, since $e^{-\alpha t}EZ^{\phi}(t)$ converges as 
$t\to\infty$, see \cite[Theorem 5.4]{Nerman}, and it is bounded in every interval $[0,t]$, as
$\sup_{s\le t}e^{-\alpha t}EZ^{\phi}(t)\le \sup_{s\le t}e^{-\alpha s}E\phi(s)\,ET(t)$.

Let $k>1$. By \eqref{alap} we have
\begin{align*}
\prod_{i=1}^k Z^{\phi}(t_i)&=\prod_{i=1}^k\phi_0(t_i)\\
&\qquad+\sum_{\emptyset\neq H\subset\{1,\dots,k\}}\bigg(
\prod_{i\notin H}\phi_0(t_i)\sum_{\substack{1\le j_i\le \xi_0(t_i),\\
\forall i\in H}}\ \prod_{i\in H} Z^{\phi}_{j_i}(t_i-\sigma_{j_i})\bigg).
\end{align*}
Here the quantities $\prod_{i\notin H}\phi_0(t_i)$ and 
$\big( Z^{\phi}_{j_i}(t_i-\sigma_{j_i}),\ i\in H\big)$ are conditionally
independent given $\xi_0$. Therefore, by fixing $H\ne\emptyset$ and taking 
conditional expectation at first, we obtain
\begin{align*}
e^{-\alpha(t_1+\dots+t_k)}&E\Bigg[
\prod_{i\notin H}\phi_0(t_i)\sum_{\substack{1\le j_i\le \xi_0(t_i),\\
\forall i\in H}}\ \prod_{i\in H} Z^{\phi}_{j_i}(t_i-\sigma_{j_i})\Bigg]\\
&=E\Bigg[E\bigg(\prod_{i\notin H}e^{-\alpha t_i}\phi_0(t_i)\biggm|\xi_0\bigg)
\\
&\qquad\times\sum_{\substack{1\le j_i\le \xi_0(t_i),\\
\forall i\in H}}\ \prod_{i\in H}e^{-\alpha \sigma_{j_i}}\ E\bigg(
\prod_{i\in H} e^{-\alpha(t_i-\sigma_{j_i})}Z^{\phi}_{j_i}(t_i-\sigma_{j_i})
\biggm|\xi_0\bigg)\Bigg].
\end{align*}
For $H=\{1,\,\dots,\,k\}$ we will separate the terms where $j_1=j_2=\dots 
=j_k=j\le\xi_0(t_1)$. Since the processes $Z_j^{\phi}(t-\sigma_j)$ are 
conditionally independent given $\xi_0$, provided the indices $j$ are all
different, in all remaining cases the conditional expectation in the last 
sum is of the following form: 
$\prod_{i=1}^{k-1}[M_i(\dots)]^{\nu_i}$, 
where $\nu_i\ge 0$ and $\sum_{i=1}^{k-1}i\nu_i=|H|$.
By the induction hypothesis, these terms can be estimated by
\[
\varrho_k=\max\left\{\, \prod_{i=1}^{k-1}C_i^{\nu_i}:
\nu_i\ge 0\text{ and }\sum_{i=1}^{k-1}i\nu_i=|H|\, \right\}\ge 1.
\]
Consequently,
\begin{equation}\label{Mk}
M_k(t,\mathbf{h})=E\Bigg[\sum_{j=1}^{\xi_0(t)}e^{-k\alpha\sigma_j}
M_k(t-\sigma_j,\mathbf{h})\Bigg]+R,
\end{equation}
where
\begin{align*}
R&\le E\Bigg[E\bigg(\prod_{i=1}^k e^{-\alpha t_i}\phi_0(t_i)\biggm|\xi_0\bigg)\\
&\quad+\varrho_k\sum_{\emptyset\ne H\subset\{1,\dots,k\}}E\bigg(\prod_{i\notin H}
e^{-\alpha t_i}\phi_0(t_i)\biggm|\xi_0\bigg)\sum_{\substack{1\le j_i\le \xi_0(t_i),\\
\forall i\in H}}\ \prod_{i\in H}e^{-\alpha \sigma_{j_i}}\Bigg]
\end{align*}
The first term on the right-hand side of \eqref{Mk} corresponds to the case where
$H=\{1,\,\dots,\,k\}$ and  $j_1=j_2=\dots =j_k$, and the remainder $R$ to all remaining cases.
The last sum can be treated in the following way.
\[
\sum_{\substack{1\le j_i\le \xi_0(t_i),\\ \forall i\in H}}\ 
\prod_{i\in H}e^{-\alpha \sigma_{j_i}}=
\prod_{i\in H}\ \sum_{j=1}^{\xi_0(t_i)}e^{-\alpha \sigma_j}
=\prod_{i\in H}\ \int_0^{t_i}e^{-\alpha s}\xi_0(\d s)\le
\big({}_{\alpha}\xi_0(\infty)\big)^{|H|}.
\]
By the conditional H\"older inequality we have
\[
E\bigg(\prod_{i\notin H}e^{-\alpha t_i}\phi_0(t_i)\biggm|\xi_0\bigg)\le
\bigg[\prod_{i\notin H}E\Big(e^{-k\alpha t_i}\phi_0(t_i)^k\Bigm|\xi_0\Big)
\bigg]^{\!1/k}.
\]
Again by H\"older, 
\begin{multline*}
E\Bigg(\bigg[\prod_{i\notin H}E\Big(e^{-k\alpha t_i}\phi_0(t_i)^k\Bigm|
\xi_0\Big)\bigg]^{\!1/k}\big({}_{\alpha}\xi_0(\infty)\big)^{|H|}\Bigg)\\
\le{\|{}_{\alpha}\xi_0(\infty)\|}_k^{|H|}
\prod_{i\notin H} e^{-\alpha t_i}{\|\phi(t_i)\|}_k\le A^{|H|}B^{k-|H|},
\end{multline*}
thus
\[
R\le \varrho_k\sum_{H\subset\{1,\dots,k\}} A^{|H|}B^{k-|H|}=\varrho_k (A+B)^k.
\]
The first term on the right-hand side of \eqref{Mk} can be written as an 
integral.
\begin{align*}
E\Bigg[\sum_{j=1}^{\xi_0(t)}e^{-k\alpha\sigma_j}
M_k(t-\sigma_j,\mathbf{h})\Bigg]&=
E\Bigg[\int_0^t e^{-k\alpha s}M_k(t-s,\mathbf{h})\,\xi(\d s)\Bigg]\\
&=\int_0^t e^{-k\alpha s}M_k(t-s,\mathbf{h})\,\mu(\d s).
\end{align*}
Define
\[ 
m=\int_0^{\infty} e^{-k\alpha t}\mu(\d t)=
\int_0^{\infty} e^{-(k-1)\alpha t}{}_{\alpha}\mu(\d t).
\]
Clearly, $m<1$, because $_{\alpha}\mu(\d t)$ is a probability measure. Let
\[
\tilde\mu(t)=m^{-1}\int_0^t e^{-k\alpha s}\mu(\d s),
\]
then $\tilde\mu(\d t)$ is also a probability measure. Using all these we arrive 
at the following inequality:
\begin{equation}\label{iter}
M_k(t,\mathbf{h})\le m\int_0^t M_k(t-s,\mathbf{h})\,\tilde\mu(\d s)+\varrho_k(A+B)^k.
\end{equation}
Since $M_k(t,\mathbf{h})$ is nonnegative and finite, standard methods of renewal theory can be applied to show that
\begin{equation}\label{end}
M_k(t,\mathbf{h})\le \frac{\varrho_k(A+B)^k}{1-m}
\end{equation}
holds for every $t$ and $\mathbf{h}$, thus the constant on right-hand side of \eqref{end} can serve as $C_k$. 
Indeed, by iterating inequality \eqref{iter} we have
\begin{multline*}
M_k\le m\, M_k*\tilde\mu+\gamma\le m^2M_k*\tilde\mu*\tilde\mu+\gamma+\gamma m\tilde\mu\le \dots \\
\le m^n M_k*\tilde\mu^{(n)}+
\gamma \sum_{i=0}^{n-1}m^i \tilde\mu^{(i)}
\le m^n\,M_k*\tilde\mu^{(n)}+ \gamma \sum_{i=0}^{\infty}m^i,
\end{multline*}
for every positive integer $n$, where $*$ denotes convolution, $\tilde\mu^{(n)}$ stands for
convolution power, and $\gamma=\varrho_k(A+B)^k$. Letting $n$ tend to infinity we get
\eqref{end}.
\end{proof}

By combining Propositions \ref{T(t)}--\ref{Zphi2} with Theorem \ref{thm1} we get
\begin{corollary}\label{cor1}
The uniformly $L_k$ boundedness
\[
\sup_{t\ge 0} e^{-\alpha t}{\|Z^{\phi}(t)\|}_k<\infty
\]
holds, if any of the following sets of conditions fulfils.
\begin{equation}\label{cond1}
{\|\xi(0)\|}_k<1,\quad {\|{}_\alpha\xi(\infty)\|}_k<\infty,\quad 
{\|\sup\nolimits_{t\ge 0}e^{-\alpha t}\phi(t)\|}_k<\infty,
\end{equation}
or
\begin{equation}\label{cond2}
{\|\xi(0)\|}_p<1,\ {\|{}_\alpha\xi(\infty)\|}_p<\infty\text{ for some }p>k,\quad
\sup\nolimits_{t\ge 0}e^{-\alpha t}{\|\phi(t)\|}_k<\infty.
\end{equation}
\end{corollary}

\section{Application: a branching random graph process}

Now, let us turn our attention to the problem of the asymptotic behaviour of the maximal degree
in the collaboration model. This randomly evolving graph process was introduced in \cite{MR17}, 
and it is defined as follows.

At the beginning, there are two vertices connected with a single edge. From this initial state, the 
graph changes in continuous time; its evolution is governed by a CMJ process defined on the edges. 
More precisely, for every edge there is a homogeneous Poisson process with unit density which rules 
its reproduction. These Poisson processes are independent. At every reproduction event a 
new vertex is added to the graph, and it is connected to one or both endpoints of the parent edge. 
Thus, the edge reproduction process, denoted by $\xi(.)$, is in fact a compound Poisson process with 
jumps at every birth event, and the jump size is equal to $2$ with probability $p$, and $1$ with probability $q = 1-p$.

Furthermore, as usually with CMJ processes, death (which now means deletion) is a feature of the 
edges. The lifetime of an edge is in strong connection with its fertility. To describe the connection 
 -- beside the naturally interpretable physical age -- we define the biological age of an edge as the 
 number of its offspring up to the moment; then the hazard rate of the life length at physical age 
 $t$ and biological age $\xi(t)$ is equal to $b+c\,\xi(t)$, where $b$ and $c$ are positive constants. 
 It is important to note that at the death only the edge is deleted but not its endpoints: once a 
 vertex is added to the graph, it will remain in the process thereafter, possibly isolated.

In \cite{MR17} and \cite{MR18}, we dealt with some asymptotic properties of this process, 
moreover, we introduced its discrete, and generalized versions, too. For example, we have described
the asymptotic behaviour of the number of living edges, the biological age, or the ratio of the vertices 
and edges. In addition, we found an interesting, apparently contradictory difference between the tails 
of the physical and the biological age distributions. However, for the analysis of the maximal degree, 
we do not need all of those results. What we will use is summarized below.

In what follows we suppose that the CMJ process of the edges is 
supercritical with Malthusian parameter $\alpha$ satisfying
\begin{equation}\label{Malthus2}
\frac{1+p}{c}\int_0^1(1-u)^{\tfrac{\alpha+1+b}{c}-1}\exp\Big(
\frac{u(2-pu)}{2c}\Big)du=1.
\end{equation}

In \cite{MR17} we have seen that the degree process of a fixed vertex 
can also be described as a CMJ process with a random number of ancestors 
(individuals born at time $0$), namely, the process starts with a single 
ancestor with probability $q$, and with two ancestors with probability 
$p=1-q$. If a vertex is 
born with initial degree $1$, its degree process is a CMJ process defined 
by the pair $(\eta(\cdot),\,\lambda)$. Here the reproduction process is
\[
\eta(t)=\xi(t)-\pi(t\wedge\lambda)+\sum_{i=1}^{2\pi(t\wedge\lambda)
-\xi(t)}\delta_i,  
\]
where $(\xi(\cdot),\pi(\cdot),\lambda)$ are the reproduction process, the 
process of birth events, and the random lifetime, resp., that define the 
CMJ process of the edges, the random variables $\delta_1,\,\delta_2,\,\dots$
are i.i.d. with $P(\delta_i=0)=P(\delta_i=1)=1/2$, and they are
independent of $\pi(\cdot)$, $\xi(\cdot)$, and $\lambda$.

If a vertex is born with initial degree $2$, its degree process is the sum 
(superposition) of two independent CMJ processes introduced above.

It is proved in \cite[Theorem 6.1]{MR17}, that the degree 
process of a vertex is supercritical if and only if 
\begin{equation*}
E\eta(\infty)=\frac{1+p}{2c}\int_0^1 (1-u)^{\tfrac{1+b}{c}-1}\exp\Big(
\frac{u(2-pu)}{2c}\Big)du>1.
\end{equation*}
In the supercritical case the Malthusian parameter $\beta$ is the only
positive root of the equation
\[
\frac{1+p}{c}\int_0^1(1-u)^{\tfrac{\beta+1+b}{c}-1}\exp\Big(
\frac{u(2-pu)}{2c}\Big)du=2,
\]
and the probability of extinction (when the monitored vertex eventually
becomes isolated) is equal to $pz^2+qz$, where $z$ is the smallest
positive root of the equation
\begin{equation*}
\frac{1+p}{2c}\int_0^1{(1-u)\strut}^{\tfrac{1+b}{c}-1}
\exp\Big(\frac{1+p}{2c}\,u(1-u)z\Big)du=1.
\end{equation*}
Clearly, $\beta<\alpha$.

Let $\tau_i$ denote the birth time of the $i{\text{th}}$ vertex (labelled in 
birth order starting from zero), $D_i(t)$ its degree and $M(t)$ the 
maximal degree in the graph at time $t$. Clearly,
\[
M(t)=\sup_{0\le i}D_i(t-\tau_i)=\max_{0\le i\le V(t)}D_i(t-\tau_i),
\]
where $V(t)$ is the number of vertices at time $t$.

In the next section we will show that the maximal degree has the 
same order of magnitude as the individual degree processes $D_i(t)$.

\section{Maximal degree}

Let us generalize the problem. For a fixed vertex $i$ consider the CMJ process 
$D_i(\cdot)$ counted by the random characteristic $\phi(\cdot)$. This process 
will be denoted by $D_i^{\phi}(\cdot)$. Furthermore, let
\[
M^{\phi}(t)=\sup_{0\le i}D_i^{\phi}(t-\tau_i)=
\max_{0\le i\le V(t)}D_i^{\phi}(t-\tau_i),
\]
the maximum of the current value of $D^{\phi}$ over all existing vertices.
Let $\zeta_i(t)=e^{-\beta t}D_i^{\phi}(t)$. Suppose that 
$\sup_{t\ge 0}e^{-\beta t}{\|\phi(t)\|}_k<\infty$. From the general theory of CMJ 
processes we know that $\zeta_i(t)\to Y_i$ almost surely, as $t\to\infty$, 
where $Y_i$, $i=0,\,1,\,\dots$ are indentically distributed random 
variables, and $Y_i>0$ almost everywhere 
on the non-extinction event of the corresponding degree process, i.e., on 
the event that vertex $i$ does not eventually become isolated. 

Moreover, $Y_i\in L_k$, because the process $D^{\phi}(t)$ corresponding to a 
vertex born at time zero can be estimated by the sum of two CMJ degree processes 
counted by the random characteristic $\phi$, and so ${}_\beta\eta(\infty)\le\eta(\infty)\le
\xi(\infty)\in L_k$. Therefore condition \eqref{cond2} is satisfied, thus 
Corollary \ref{cor1} implies that the process $\zeta(t)=e^{-\beta t}
D^{\phi}(t)$ is bounded in $L_k$. Hence $Y=\lim_{t\to\infty}\zeta(t)\in L_k$.

\begin{theorem}\label{max}

Suppose $\sup_{t\ge 0}e^{-\beta t}{\|\phi(t)\|}_k<\infty$ for some $k>\alpha/\beta$.
Then 
\begin{equation}\label{L_k}
e^{-\beta t}M^{\phi}(t)\to\sup_{i\ge 0} e^{-\beta\tau_i}Y_i
=:\Delta
\end{equation}
in $L_k$.  The limit is positive on the event of 
non-exhaustion of the graph. Moreover,
\begin{equation}\label{liminf}
\liminf_{t\to\infty}e^{-\beta t}M^{\phi}(t)=\Delta
\end{equation}
almost surely.
\end{theorem}
\begin{proof}
First of all, we will show that $\Delta\in L_k$. Using that $Y_i$ and $\tau_i$ 
are independent, we get
\begin{multline*}
E(\Delta^k)\le E\bigg[\sum_{i\ge 0} e^{-k\beta\tau_i}Y_i^k\bigg]
=E(Y_1^k)E\bigg[\sum_{i\ge 0} e^{-k\beta\tau_i}\bigg]\\
=E(Y_1^k)E\bigg[\int_0^{\infty}e^{-k\beta t}V(dt)\bigg].
\end{multline*}
By Fubini's theorem we have
\begin{align*}
E\bigg[\int_0^{\infty}e^{-k\beta t}V(dt)\bigg]&=
E\bigg[\int_0^{\infty}\int_t^{\infty}k\beta e^{-k\beta s}ds\,V(dt)\bigg]\\
&=E\bigg[\int_0^{\infty}k\beta e^{-k\beta s}\int_0^s V(dt)\,ds\bigg]\\
&=E\bigg[\int_0^{\infty}k\beta e^{-k\beta s}V(s)\,ds\bigg]\\
&=\int_0^{\infty}k\beta e^{-k\beta s}EV(s)\,ds.
\end{align*}
Here $EV(s)\le C\,e^{\alpha s}$, thus
\[
E(\Delta^k)\le C\,E(Y_1^k)\int_0^{\infty}k\beta e^{-(k\beta-\alpha)s}ds
<\infty,
\]
if $k\beta>\alpha$.

Next, let us estimate the $L_k$-distance of $e^{-\beta t}M^{\phi}(t)$ 
and $\Delta$. Clearly,
\begin{align*}
\big\|e^{-\beta t}M^{\phi}(t)-\Delta\big\|_k
&\le \Big\|e^{-\beta t}M^{\phi}(t)-\max_{i\le n}e^{-\beta t}
D_i^{\phi}(t-\tau_i)\Big\|_k\\
&+\Big\|\max_{i\le n}e^{-\beta t}D_i^{\phi}(t-\tau_i)-
\max_{i\le n}e^{-\beta \tau_i}Y_i\Big\|_k\\
&+\Big\|\Delta-\max_{i\le n}e^{-\beta \tau_i}Y_i\Big\|_k\\
&=: Q_1+Q_2+Q_3.
\end{align*}

Here $Q_1$ and $Q_3$ can be estimated uniformly in $t$.
\begin{align*}
Q_1^k &= E\bigg[e^{-\beta n}\Big(M^{\phi}(n)-
\max_{0\le i\le n}D_i^{\phi}(n-\tau_i)\Big)\bigg]^{\!k}\\
&\le E\bigg[\sum_{i>n}e^{-k\beta n}\big(D_i^{\phi}(t-\tau_n)\big)^k\bigg]\\
&=E\bigg[\sum_{i>n}e^{-k\beta \tau_i}\zeta_i^k(n-\tau_i)\bigg]
\end{align*}
Take conditional expectation inside of the sum with respect to $\tau_i$, 
and use that $\zeta_i(\cdot)$ is independent of $\tau_i$.
\begin{align*}
E\bigg[\sum_{i>n}e^{-k\beta \tau_i}\zeta_i^k(t-\tau_i)\bigg]
&=E\bigg[\sum_{i>n}e^{-k\beta \tau_i}E\big(\zeta_i^k(t-\tau_i)
\bigm|\tau_i\big)\bigg]\\
&\le \sup_{t\ge 0}E\big(\zeta_i^k(t)\big)\,E\bigg[\sum_{i>n}
e^{-k\beta \tau_i}\bigg].
\end{align*}
The first expectation on the right-hand side is finite by 
Corollary \ref{cor1}. The second expectation can be treated similarly 
to what we have done when $E(\Delta^k)$ was estimated.
\begin{align*}
E\bigg[\sum_{i>n}e^{-k\beta \tau_i}\bigg]&=
E\bigg[\int_0^{\infty}\ind(V(t)>n)e^{-k\beta t}V(dt)\bigg]\\
&=E\bigg[\int_0^{\infty}k\beta e^{-k\beta s}\int_0^s \ind(V(t)>n)V(dt)\,
ds\bigg]\\
&=E\bigg[\int_0^{\infty}k\beta e^{-k\beta s}(V(s)-n)^+\,
ds\bigg]\\
&=\int_0^{\infty}k\beta e^{-k\beta s}E(V(s)-n)^+\,ds
\end{align*}

Let us turn to $Q_3$. Similarly, we have
\begin{align*}
Q_3^k &= E\bigg[\sup_{i\le 0}e^{-\beta \tau_i}Y_i-
\max_{i\le n}e^{-\beta \tau_i}Y_i\bigg]^{\!k}\\
&\le E\bigg[\sum_{i>n} e^{-k\beta\tau_i}Y_i^k\bigg]\\
&=E(Y_1^k)\,E\bigg[\sum_{i>n} e^{-k\beta\tau_i}\bigg]\\
&\le \sup_{t\ge 0}E\big(\zeta_i^k(t)\big)\,E\bigg[\sum_{i>n}
e^{-k\beta \tau_i}\bigg],
\end{align*}
which has already been estimated above.

Finally, if $n$ is fixed, $\max_{i\le n}e^{-\beta t}D_i^{\phi}(t-\tau_i)
\to\max_{i\le n}e^{-\beta \tau_i}Y_i$ in $L_k$, as $t\to\infty$,
hence $Q_2\to 0$. Thus, for every positive integer $n$,  
\[
\limsup_{t\to\infty}\big\|e^{-\beta t}M^{\phi}(t)-\Delta\big\|_k\le
2\sup_{t\ge 0}{\|\zeta_i(t)\|}_k \bigg[\int_0^{\infty}k\beta 
e^{-k\beta s}E(V(s)-n)^+\,ds\bigg]^{\!1/k}.
\]
The integrand on the right-hand side is majorized by 
\[
k\beta e^{-k\beta s}EV(s)\le C\,k\beta e^{-(k\beta-\alpha)s},
\]
and it tends to $0$ pointwise as $n\to\infty$. Hence the dominated convergence 
theorem can be applied to complete the proof of \eqref{L_k}.

For the proof of \eqref{liminf} notice that
\begin{multline*}
\liminf_{t\to\infty}e^{-\beta t}M^{\phi}(t)\ge\liminf_{t\to\infty}\ 
\max_{i\le n}e^{-\beta t}D_i^{\phi}(t)
=\max_{i\le n}\ \lim_{t\to\infty} e^{-\beta t}D_i^{\phi}(t)\\
=\max_{i\le n}\ \lim_{t\to\infty}e^{-\beta \tau_i}\zeta_i(t-\tau_i)
=\max_{i\le n}e^{-\beta\tau_i}Y_i
\end{multline*}
for every positive integer $n$. Hence it follows that
$\liminf_{t\to\infty}e^{-\beta t}M^{\phi}(t)\ge\Delta$.
Comparing this with \eqref{L_k} we obtain \eqref{liminf}.
\end{proof}

Choosing $\phi(t)=\ind(t<\lambda)$ we get $D^{\phi}(t)=D(t)$, the 
degree process, and $M^{\phi}(t)=M(t)$, the maximal degree in the 
graph at time $t$.
\begin{corollary}\ 

$\liminf\limits_{t\to\infty}e^{-\beta t}M(t)=\Delta$
almost surely, and $\lim\limits_{t\to\infty}e^{-\beta t}M(t)=\Delta$ 
in $L_k$ for all $k\ge 1$.
\end{corollary}

We remark that a similar limit theorem is valid if the degree of a vertex is counted
by including the deceased, deleted edges. (We can think that all edges 
are red at the moment of birth, and when an edge dies, it is not deleted 
but recolored blue instead. Blue edges stop reproducing, but they  
still matter when degrees are counted.)

\section{Conclusions}

In this paper, we returned to the analysis of the collaboration model, introduced in \cite{MR17}. 
The main problem considered here was the order of magnitude of its maximal degree. We 
presented certain $L_k$ boundedness conditions on the reproduction process and the random 
characteristics that proved to be sufficient for a CMJ process to have finite $k$th 
moment (uniformly in $t \geq 0$). Then, by applying these results we concluded that the 
maximal degree and a single vertex's degree are of the same order of magnitude.

However, a number of questions remained open. It is clear that our results on the maximal 
degree can be strengthened in the sense that, in addition to the $L_k$ convergence proved
here, almost sure convergence of the suitably normed maximal degree is (almost) sure to hold. 
Hopefully, to this end it would be enough to show that requiring the same conditions on
the reproduction and the random characteristics, not just the supremum of the $k$th 
moments, but the $k$th moment of the supremum of the counted CMJ process is 
bounded. This problem appears manageable using some martingale theory, and we are planning 
to return to it in a forthcoming paper.

\end{document}